\documentclass[11pt]{article}
\usepackage[a4paper,margin=1.05in]{geometry}
\usepackage{amsmath,amssymb,amsthm,mathtools}
\usepackage{microtype}
\usepackage{enumitem}
\usepackage{booktabs}
\usepackage{xcolor}
\usepackage{hyperref}
\usepackage[nameinlink,noabbrev]{cleveref}
\usepackage{mathrsfs}
\usepackage{bm}
\usepackage{array}

\hypersetup{
  colorlinks=true,
  linkcolor=blue!55!black,
  citecolor=blue!55!black,
  urlcolor=blue!55!black,
  pdftitle={Contraction Algebras, Cyclotomic Forms, and Reflection Rigidity},
  pdfauthor={Xiaobin Li}
}

\newtheorem{theorem}{Theorem}[section]
\newtheorem{proposition}[theorem]{Proposition}
\newtheorem{lemma}[theorem]{Lemma}
\newtheorem{corollary}[theorem]{Corollary}
\theoremstyle{definition}
\newtheorem{definition}[theorem]{Definition}
\newtheorem{assumption}[theorem]{Assumption}
\theoremstyle{remark}
\newtheorem{remark}[theorem]{Remark}

\newcommand{\supp}{\operatorname{Supp}}
\newcommand{\ord}{\operatorname{ord}}
\newcommand{\Span}{\operatorname{span}}
\newcommand{\Tr}{\operatorname{tr}}

\newcommand{\con}{\mathrm{con}}
\newcommand{\ab}{\mathrm{ab}}
\newcommand{\wt}{\mathsf W}
\newcommand{\acoeff}{\mathsf A}
\newcommand{\C}{\mathbb C}
\newcommand{\R}{\mathbb R}
\newcommand{\Z}{\mathbb Z}

\newcommand{\Q}{\mathbb Q}
\newcommand{\dd}{\,\mathrm d}

\title{\bfseries Contraction Algebras, Cyclotomic Forms, and Reflection Rigidity}
\author{Xiaobin Li\\[2pt]
\small School of Mathematics, Southwest Jiaotong University,\\ \small
West Zone, High-Tech District, Chengdu, Sichuan 611756, China}
\date{}

\begin{document}
\maketitle

\begin{abstract}
We associate to finite block size data
\[
\bm n=(n_d)_{d\ge1},\qquad n_1>0,
\]
a positive definite cyclotomic form and a Householder reflection representation.  Writing
\[
N=\operatorname{lcm}\{d:n_d\ne0\},\qquad \wt=\sum_d d^2n_d,
\]
the form on the cyclic group $C_N$ is
\[
B_{\bm n}(g,h)=\frac1{\wt}\sum_{(g^{-1}h)^d=1}d^2n_d.
\]
Its labelled Fourier spectrum reconstructs the complete multiplicity sequence by M\"obius inversion.  If the $n_d$ are the Wedderburn multiplicities of a split semisimple algebra $S=\prod_d M_d(k)^{n_d}$, the divisibility widths are dimensions of characteristic sums of Wedderburn blocks, and the form together with $\dim_kS$ recovers the isomorphism type of $S$.

The main theorem is an arithmetic rigidity result.  For $N\ge2$, the reflection group is finite if and only if it preserves a full rank lattice, equivalently
\[
\supp(\bm n)=\{1,N\},\qquad n_1=N^2n_N.
\]
In this case $B_{\bm n}=\frac12(I+J)$ and the group is $W(A_N)\cong S_{N+1}$.  In semisimple algebra terms, finiteness is equivalent to
\[
S\cong k^{N^2m}\times M_N(k)^m
\]
for some $m\ge1$.  For $N=2$, every nonfinite case is a dense copy of $D_\infty$ in the corresponding orthogonal group.

For contraction algebras of threefold flops, Toda's width formulas identify the total scale and primitive Fourier mode with $\dim A_{\con}$ and $\dim A_{\con}^{\ab}$.  This yields infinitude for all irreducible flops of length at least three and, in length two,
\[
\Gamma_f\text{ finite}\quad\Longleftrightarrow\quad
\dim A_{\con}=2\dim A_{\con}^{\ab}.
\]
The final part gives a complementary BPS/Riemann--Hilbert realization of the same reflection matrices.
\end{abstract}


\section{Introduction and main results}

Finite dimensional algebras arising from threefold flops provide a natural setting in which block size data, divisor arithmetic, and reflection theory interact.  Contraction algebras, introduced by Donovan and Wemyss, are finite dimensional noncommutative invariants of three dimensional flopping contractions \cite{DWDuke}.  Their dimensions and abelianizations are tied to local genus zero Gopakumar-Vafa invariants through Toda's width formulas \cite{TodaWidth}, while recent semisimple deformation results identify the relevant Gopakumar-Vafa multiplicities with Wedderburn block multiplicities in types $A$ and $D$ \cite{JS2026}.  This motivates asking whether the block-size data themselves carry a canonical algebraic structure beyond their numerical role.

The purpose of this paper is to construct such a structure.  From a finite multiplicity sequence $\bm n=(n_d)$ we form characteristic divisibility widths, a positive definite cyclic Gram form, and a Householder reflection representation.  The principal result is an arithmetic rigidity theorem: finiteness of the reflection group is equivalent to preservation of a full rank lattice and forces a unique $A_N$-type block pattern.  For a split semisimple algebra $S=\prod_dM_d(k)^{n_d}$, the divisibility widths are dimensions of characteristic two-sided ideals formed from Wedderburn blocks; the resulting pointed cyclotomic form, together with $\dim_kS$, recovers the complete Wedderburn type.  Thus the construction is intrinsic at the level of split semisimple algebras and becomes a natural shadow of contraction algebra data.

The construction is deliberately coarser than the full contraction algebra.  Brown and Wemyss exhibited distinct flops with identical Gopakumar-Vafa spectra but different contraction algebras \cite{BrownWemyss}; such flops necessarily have the same cyclotomic form.  The new invariant should therefore be viewed as the part of the contraction algebra structure visible through its spectrum of matrix-block multiplicities.  Conversely, once the total width is supplied, the labelled Fourier spectrum reconstructs all multiplicities by M\"obius inversion.

There are existing reflection and braid group structures attached to flops; see \cite{Humphreys} for general reflection group background.  Donovan and Wemyss constructed hyperplane arrangement and braid group actions from sequences of threefold flops \cite{DWBraids}, and related categorical viewpoints were developed further in \cite{BodzentaBondal}.  Coxeter symmetries acting on K\"ahler and curve classes were studied in \cite{LukasRuehle,KuuselaMcGovern}.  The present construction has a different source: the reflecting vectors and their positive form are reconstructed directly from finite block size/Gopakumar-Vafa data rather than from geometric flop transformations.

The final part of the paper gives a complementary analytic realization.  Exact chargewise BPS difference identities lead to a hyperbolic interaction whose associated one-cut singular integral equation vectorizes to a cyclotomic Riemann-Hilbert problem with the same homogeneous reflection matrices.  We compare this linear problem with Bridgeland's Donaldson-Thomas Riemann-Hilbert theory \cite{BridgelandDT,BridgelandConifold}; direct tangent equivalence is obstructed, but each reflection factors as a unipotent tangent shear followed by a rank one folding involution.  These analytic results are logically independent of the algebraic rigidity theorem.  Related thermodynamic and difference equation motivations appear in \cite{LiYagi2021,LiDiffEq2026}.

The main contributions are therefore algebraic: a cyclotomic positive form attached to Wedderburn multiplicities, reconstruction of the block size data from the labelled spectrum, and a finite/lattice rigidity classification of the associated reflection group.  The Fourier diagonalization itself is standard for circulant matrices \cite{DavisCirculant,SoIntegral}; the new content is the algebraic source of the form and the rigidity forced by its divisibility structure.

\subsection*{The cyclotomic metric shadow}
Let
\[
\bm n=(n_d)_{d\ge1},\qquad n_d\in\Z_{\ge0},\qquad n_1>0,
\]
be finitely supported.  Put
\[
N=\operatorname{lcm}\{d:n_d>0\},\qquad
\wt(\bm n)=\sum_d d^2n_d.
\]
For the cyclic group $C_N$, define
\begin{equation}\label{eq:intro-kernel}
B_{\bm n}(g,h)=\frac1{\wt(\bm n)}\sum_{(g^{-1}h)^d=1}d^2n_d.
\end{equation}
Our first result is that this is positive definite and therefore defines a canonical cyclic metric representation.  The diagonalization of a circulant kernel by cyclic characters is standard (see, for example, \cite{DavisCirculant,SoIntegral}); the point here is the specific arithmetic kernel forced by the spectrum and the rigidity consequences below.  Its Fourier spectrum is
\begin{equation}\label{eq:intro-fourier}
\lambda_k=\frac1{\wt(\bm n)}\sum_{d\mid\gcd(k,N)}d^3n_d,
\end{equation}
and the divisor labelled spectrum recovers the $n_d$ by M\"obius inversion.

\begin{theorem}[Metric and spectral reconstruction]\label{thm:intro-metric}
The kernel \eqref{eq:intro-kernel} is positive definite.  The corresponding cyclic metric shadow is canonical up to pointed $C_N$-equivariant isometry.  Its Fourier eigenvalues are given by \eqref{eq:intro-fourier}.  If $\Lambda_r$ denotes the eigenvalue attached to the gcd-class $r\mid N$, then
\[
n_d=\frac{\wt(\bm n)}{d^3}\sum_{e\mid d}\mu(d/e)\Lambda_e.
\]
Consequently the labelled Fourier spectrum together with $\wt(\bm n)$ determines the complete spectrum $\bm n$; the normalized Gram form alone determines its projective class.
\end{theorem}

If the multiplicities $n_d$ are the Wedderburn multiplicities of a split finite-dimensional semisimple algebra
\[
S\cong\prod_d M_d(k)^{n_d},
\]
then the construction is an isomorphism invariant of $S$.  The divisibility widths are dimensions of characteristic sums of Wedderburn blocks, and the pointed metric shadow together with $\dim_k S$ recovers the complete isomorphism type of $S$.  Thus the numerical construction has a purely algebraic interpretation even before contraction algebras are introduced.

The extra factor of $d$ in \eqref{eq:intro-fourier} has a simple representation theoretic meaning: $d^2$ is the matrix block weight naturally appearing in contraction algebra width formulas, while $d=|C_d|$ is the order of the relevant cyclotomic stabilizer.

\subsection*{Reflection rigidity}
Let $v_g$ be the unit orbit vectors of the cyclotomic metric shadow and define
\[
R_g(x)=x-2\langle x,v_g\rangle v_g,
\qquad
\Gamma_{\bm n}=\langle R_g:g\in C_N\rangle.
\]
The principal rigidity theorem is the following.

\begin{theorem}[Cyclotomic reflection rigidity]\label{thm:intro-rigidity}
Assume $N\ge2$.  The following conditions are equivalent:
\begin{enumerate}[label=(\roman*)]
\item $\Gamma_{\bm n}$ is finite;
\item $\Gamma_{\bm n}$ preserves a full-rank lattice;
\item $\supp(\bm n)=\{1,N\}$ and $n_1=N^2n_N$;
\item the Gram matrix is $B_{\bm n}=\frac12(I_N+J_N)$;
\item $\Gamma_{\bm n}\cong W(A_N)\cong S_{N+1}$.
\end{enumerate}
Equivalently, finiteness is characterized spectrally by
\[
\lambda_0=\frac{N+1}{2},\qquad \lambda_k=\frac12\quad(1\le k<N).
\]
Thus the finite case is precisely exact equipartition of the distinguished vector among all nontrivial Fourier characters.
\end{theorem}

For $N=1$, the unique Householder reflection is multiplication by $-1$, so the reflection group is $W(A_1)\cong C_2$.  For $N=2$ the infinite case can also be classified dynamically.

\begin{theorem}[Rank two finite-or-dense dichotomy]\label{thm:intro-rank2}
Assume $N=2$ and $n_1,n_2>0$.  If $n_1=4n_2$, then $\Gamma_{\bm n}\cong W(A_2)\cong S_3$.  Otherwise $\Gamma_{\bm n}\cong D_\infty$, and its image is dense in the compact orthogonal group of the cyclotomic metric.
\end{theorem}

\subsection*{Threefold flops}
Let $f:X\to Y$ be a smooth irreducible threefold flop with exceptional curve $C\cong\mathbb P^1$, length $\ell$, and local genus zero Gopakumar-Vafa spectrum $n_j=n_{j[C]}$.  Toda proved \cite{TodaWidth}
\begin{equation}\label{eq:intro-toda}
\dim A_{\con}=\sum_{j=1}^{\ell}j^2n_j,
\qquad
\dim A_{\con}^{\ab}=n_1,
\end{equation}
and records $n_j\ge1$ for $1\le j\le\ell$.  Thus the cyclotomic order equals $\operatorname{lcm}(1,\ldots,\ell)$.

\begin{theorem}[Flop consequences]\label{thm:intro-flop}
Let $f$ be as above.
\begin{enumerate}[label=(\alph*)]
\item The primitive Fourier eigenvalue is
\[
\lambda_{\rm prim}=\frac{\dim A_{\con}^{\ab}}{\dim A_{\con}}.
\]
\item If $\ell=1$, then $\Gamma_f\cong W(A_1)$.
\item If $\ell\ge3$, then $\Gamma_f$ is infinite.
\item If $\ell=2$, then
\[
\Gamma_f\text{ is finite}
\quad\Longleftrightarrow\quad
\dim A_{\con}=2\dim A_{\con}^{\ab}.
\]
In the finite case $\Gamma_f\cong W(A_2)$; otherwise $\Gamma_f$ is a dense copy of $D_\infty$ in the corresponding two-dimensional orthogonal group.
\end{enumerate}
\end{theorem}

Using the 2026 theorem of Jelisiejew-Smoktunowicz \cite{JS2026}, we obtain a complete type $D$ application: among the noncommutative $D_{n,m}$ contraction algebras covered by their semisimple deformation theorem, $D_{2,2}$ is the unique finite case, with spectrum $(n_1,n_2)=(4,1)$.

\subsection*{BPS and Riemann--Hilbert realization}
Motivated in part by the role of finite-difference structures in the Nekrasov-Okounkov random partition approach \cite{NekrasovOkounkov}, we associate to each charge $d$ the factor
\[
\Phi^+_{\hbar,d}(t)=n_d\sum_{m\ge1}m\log(1-e^{-m\hbar-dt}).
\]
It satisfies an exact finite-difference identity at step $\hbar/d$.  After pairing the two charge chambers through the same entire charge factor, the normalized second difference is precisely
\[
d^2n_d\log\!\left(2\sinh\frac{d|t|}{2}\right).
\]
The resulting interaction gives, under a regular one cut assumption, a nonlocal scalar singular integral equation.  On the cyclotomic cover it becomes a vector Riemann-Hilbert problem with homogeneous jumps equal to the reflections $R_g$ above.

We stress that we do not prove a large deviation theorem deriving the variational functional from a complete microscopic Donaldson-Thomas ensemble.  The exact difference identity and the associated variational/Riemann-Hilbert realization are the analytic results established here.

\subsection*{Relation with DT Riemann-Hilbert automorphisms}
Bridgeland's ray automorphisms have unipotent tangent maps on uncoupled BPS rays, hence determinant $+1$, whereas our homogeneous jumps are reflections of determinant $-1$.  Direct tangent conjugacy is therefore impossible.  Nevertheless each reflection admits a canonical factorization
\[
R_j=U_jD_j,
\]
where $D_j$ is a rank one folding involution and the canonically determined shear $U_j$ is realized as the tangent map of a finite auxiliary uncoupled BPS automorphism whose multiplicities are the cyclotomic divisibility widths.  We refer to this as a \emph{folded tangent BPS realization}.  The auxiliary BPS structure is not asserted to be the genuine geometric DT BPS structure of the original flop.

The paper is organized as follows:  Section 2 constructs the cyclotomic metric shadow. Section 3 develops its Fourier theory and reconstruction formulas. Section 4 proves the reflection and lattice rigidity results.  Section 5 applies the abstract theory to threefold flops and contraction algebras. Section 6 derives the exact BPS difference Hessian and the associated cyclotomic Riemann-Hilbert problem. Section 7 compares the reflection jumps with Donaldson-Thomas BPS automorphisms.

\section{Cyclotomic metric shadows}\label{sec:metric}

Let $\bm n=(n_d)_{d\ge1}$ be finitely supported with $n_d\in\Z_{\ge0}$ and $n_1>0$.  Put
\[
D(\bm n)=\{d\ge1:n_d>0\},\qquad
N=\operatorname{lcm}D(\bm n),\qquad
\wt=\sum_{d\mid N}d^2n_d.
\]

\subsection{Divisibility widths}
For each divisor $r\mid N$, define
\begin{equation}\label{eq:width-r}
w_r(\bm n)=\sum_{\substack{d\mid N\\r\mid d}}d^2n_d.
\end{equation}
In particular $w_1=\wt$.  Let $G_N=C_N$.  For $g\in G_N$, define
\begin{equation}\label{eq:kappa}
\kappa_{\bm n}(g)=\frac{w_{\ord(g)}(\bm n)}{\wt}
=\frac1\wt\sum_{g^d=1}d^2n_d.
\end{equation}
Clearly $\kappa_{\bm n}(e)=1$ and $\kappa_{\bm n}(g^{-1})=\kappa_{\bm n}(g)$.

\begin{proposition}[Positive-type property]\label{prop:positive}
The matrix
\[
B_{\bm n}(g,h)=\kappa_{\bm n}(g^{-1}h)
\]
is positive definite.
\end{proposition}

\begin{proof}
For each $d\mid N$, let $H_d\le G_N$ be the subgroup of order $d$.  Consider
\[
\mathcal H_{\bm n}=\bigoplus_{d\mid N}\R[G_N/H_d]
\]
with the direct sum Euclidean structure and left permutation action $\pi$.  Put
\[
v=\frac1{\sqrt\wt}\bigoplus_{d\mid N}d\sqrt{n_d}\,\delta_{H_d}.
\]
Then $\|v\|=1$ and
\[
\langle\pi(g)v,\pi(h)v\rangle
=\frac1\wt\sum_{g^{-1}h\in H_d}d^2n_d
=\kappa_{\bm n}(g^{-1}h).
\]
Thus $B_{\bm n}$ is a Gram matrix.  The $d=1$ summand contributes $(n_1/\wt)I_N$, hence
\[
B_{\bm n}\ge \frac{n_1}{\wt}I_N>0.
\]
\end{proof}

\begin{definition}[Cyclotomic metric shadow]\label{def:shadow}
Let $V_{\bm n}=\R[G_N]$ and set
\[
\langle\delta_g,\delta_h\rangle_{\bm n}
=\kappa_{\bm n}(g^{-1}h).
\]
The left regular action is orthogonal.  With $v_{\bm n}=\delta_e$, we call
\[
\mathfrak C(\bm n)=(G_N,V_{\bm n},\pi_{\bm n},v_{\bm n})
\]
the \emph{cyclotomic metric shadow} of $\bm n$.
\end{definition}

\begin{proposition}[Pointed GNS uniqueness]\label{prop:GNS}
Suppose $(G_N,V,\pi,v)$ is cyclic and satisfies
\[
\langle\pi(g)v,\pi(h)v\rangle=\kappa_{\bm n}(g^{-1}h).
\]
Then there is a unique $G_N$-equivariant isometry $U:V_{\bm n}\to V$ with $U(v_{\bm n})=v$.
\end{proposition}

\begin{proof}
Define $U(\delta_g)=\pi(g)v$.  Equality of matrix coefficients implies preservation of all pairwise inner products.  Positive definiteness gives injectivity, cyclicity gives surjectivity, and uniqueness follows from the cyclic orbit.
\end{proof}

\begin{remark}
The construction does not depend on a choice of primitive generator of $C_N$.  Every automorphism of $C_N$ preserves element orders and therefore preserves $\kappa_{\bm n}$.
\end{remark}

\subsection{Split semisimple algebras}\label{subsec:semisimple-abstract}
Let $k$ be a field and let $S$ be a split finite dimensional semisimple $k$-algebra.  Write its Wedderburn decomposition as
\begin{equation}\label{eq:wedderburn-abstract}
S\cong\prod_{d\in D}M_d(k)^{n_d},
\qquad n_d\in\Z_{\ge0},\qquad n_1>0,
\end{equation}
and set $N_S=\operatorname{lcm}D$.  The multiplicities $n_d$ and the multiset of matrix sizes are invariants of the $k$-algebra isomorphism type.

For every divisor $r\mid N_S$, let $S^{(r)}$ be the sum of the simple two-sided factors of $S$ whose matrix size is divisible by $r$.  Equivalently, under any Wedderburn decomposition as in \eqref{eq:wedderburn-abstract},
\[
S^{(r)}\cong\prod_{r\mid d}M_d(k)^{n_d},
\]
with zero components in the remaining factors.

\begin{proposition}[Semisimple divisibility sectors]\label{prop:semisimple-sectors}
Each $S^{(r)}$ is a characteristic two-sided ideal of $S$, and
\[
\dim_k S^{(r)}=\sum_{r\mid d}d^2n_d=w_r(\bm n).
\]
Consequently the cyclotomic kernel can be written intrinsically as
\begin{equation}\label{eq:semisimple-kernel}
\kappa_S(g)=\frac{\dim_k S^{(\ord g)}}{\dim_k S},
\qquad g\in C_{N_S}.
\end{equation}
\end{proposition}

\begin{proof}
A $k$-algebra automorphism of a split semisimple algebra permutes simple factors only among isomorphic matrix algebras.  Hence the sum of all factors whose matrix size is divisible by $r$ is invariant under every automorphism, so $S^{(r)}$ is characteristic.  The dimension formula follows from $\dim_k M_d(k)=d^2$, and \eqref{eq:semisimple-kernel} is exactly the definition of $\kappa_{\bm n}$ rewritten in terms of these dimensions.
\end{proof}

\begin{theorem}[Recovery of the Wedderburn type]\label{thm:recover-semisimple}
Let $S$ be as in \eqref{eq:wedderburn-abstract}.  The pointed cyclotomic metric shadow $\mathfrak C(S):=\mathfrak C(\bm n)$ together with $\dim_k S$ determines $S$ up to $k$-algebra isomorphism.
\end{theorem}

\begin{proof}
The cyclic group in the shadow determines $N_S$, while its labelled Fourier spectrum and $\dim_k S=\wt(\bm n)$ recover every multiplicity $n_d$ by \Cref{thm:reconstruct} below.  Artin-Wedderburn then recovers the product \eqref{eq:wedderburn-abstract} up to $k$-algebra isomorphism.
\end{proof}

\section{Fourier spectrum and reconstruction}\label{sec:fourier}

Fix a generator $\omega$ of $C_N$.  Let
\[
\chi_k(\omega^j)=e^{2\pi i kj/N},\qquad 0\le k<N,
\]
and define
\[
u_k=\frac1{\sqrt N}\sum_{j=0}^{N-1}e^{-2\pi i kj/N}\delta_{\omega^j}.
\]

\begin{theorem}[Fourier spectrum formula]\label{thm:fourier}
The vectors $u_k$ diagonalize $B_{\bm n}$, with eigenvalues
\begin{equation}\label{eq:lambda}
\lambda_k=\frac1\wt\sum_{d\mid\gcd(k,N)}d^3n_d,
\end{equation}
where $\gcd(0,N)=N$.
\end{theorem}

\begin{proof}
The matrix $B_{\bm n}$ is circulant.  Hence
\[
\lambda_k=\sum_{j=0}^{N-1}\kappa_{\bm n}(\omega^j)e^{-2\pi i kj/N}.
\]
Using \eqref{eq:kappa} and interchanging sums gives
\[
\lambda_k=\frac1\wt\sum_d d^2n_d\sum_{g\in H_d}\overline{\chi_k(g)}.
\]
The inner character sum equals $d$ if $d\mid k$ and vanishes otherwise.  This yields \eqref{eq:lambda}.
\end{proof}

\begin{corollary}\label{cor:extrema}
For every $k$,
\[
\lambda_k\ge\frac{n_1}{\wt}>0.
\]
Moreover
\[
\lambda_{\min}=\frac{n_1}{\wt},\qquad
\lambda_{\max}=\lambda_0=\frac{\sum_dd^3n_d}{\sum_dd^2n_d},
\]
so
\[
\operatorname{cond}(B_{\bm n})=\frac{\sum_dd^3n_d}{n_1}.
\]
If $\gcd(k,N)=1$, then $\lambda_k=n_1/\wt$.
\end{corollary}

\subsection{Character energy}
Let
\[
P_k=\frac1N\sum_{g\in C_N}\overline{\chi_k(g)}\pi_{\bm n}(g).
\]

\begin{theorem}[Character energy formula]\label{thm:energy}
For every $k$,
\[
\lambda_k=N\|P_kv_{\bm n}\|^2.
\]
In particular $\sum_{k=0}^{N-1}\lambda_k=N$.
\end{theorem}

\begin{proof}
Since $P_k$ is the orthogonal character projection,
\[
N\|P_kv_{\bm n}\|^2
=\sum_{g\in C_N}\overline{\chi_k(g)}\langle v_{\bm n},\pi_{\bm n}(g)v_{\bm n}\rangle,
\]
which is exactly the Fourier transform defining $\lambda_k$.  The second identity is Parseval, equivalently $\Tr B_{\bm n}=N$.
\end{proof}

\subsection{M\"obius reconstruction}
For $r\mid N$, let $\Lambda_r$ denote the common eigenvalue for the gcd-class $r$; explicitly $\Lambda_N=\lambda_0$ and
\[
\wt\Lambda_r=\sum_{d\mid r}d^3n_d.
\]

\begin{theorem}[Spectral reconstruction]\label{thm:reconstruct}
For every $d\mid N$,
\begin{equation}\label{eq:mobius}
n_d=\frac{\wt}{d^3}\sum_{e\mid d}\mu(d/e)\Lambda_e.
\end{equation}
Thus $(\{\Lambda_r\}_{r\mid N},\wt)$ uniquely determines $\bm n$.
\end{theorem}

\begin{proof}
Apply M\"obius inversion to the divisor sum
\[
\wt\Lambda_r=\sum_{d\mid r}a_d,
\qquad a_d=d^3n_d.
\]
\end{proof}

\begin{corollary}[Projective reconstruction]\label{cor:projective}
If $\bm n$ and $\bm n'$ have the same cyclotomic order, then $B_{\bm n}=B_{\bm n'}$ if and only if $n'_d=c n_d$ for all $d$ and some $c>0$.
\end{corollary}

\section{Reflection groups and cyclotomic rigidity}\label{sec:rigidity}

For $g\in C_N$, let $v_g=\pi_{\bm n}(g)v_{\bm n}$ and define
\begin{equation}\label{eq:reflection}
R_g(x)=x-2\langle x,v_g\rangle v_g.
\end{equation}
Let
\[
\Gamma_{\bm n}=\langle R_g:g\in C_N\rangle\subset O(V_{\bm n}).
\]
The cyclic symmetry satisfies $\pi(a)R_g\pi(a)^{-1}=R_{ag}$.

\subsection{Pairwise arithmetic}
For $g\ne h$, put
\[
c(g,h)=\langle v_g,v_h\rangle=\kappa_{\bm n}(g^{-1}h)\in\Q\cap[0,1).
\]
On $P_{g,h}=\Span\{v_g,v_h\}$ the product $R_gR_h$ is a rotation.

\begin{lemma}[Rational finite order obstruction]\label{lem:finitepair}
If $R_gR_h$ has finite order, then
\[
c(g,h)\in\left\{0,\frac12\right\}.
\]
\end{lemma}

\begin{proof}
On $P_{g,h}$,
\[
\Tr(R_gR_h)=4c(g,h)^2-2.
\]
Finite order implies that this trace is an algebraic integer.  It is rational, hence an integer.  Since $0\le c<1$, the possible values are
\[
c=0,\frac12,\frac1{\sqrt2},\frac{\sqrt3}{2},
\]
and rationality eliminates the last two.
\end{proof}

\begin{lemma}[Support rigidity]\label{lem:support}
Assume
\[
\frac{w_r(\bm n)}\wt\in\left\{0,\frac12\right\}
\]
for every divisor $r>1$ of $N$.  Then
\[
\supp(\bm n)=\{1,N\},\qquad n_1=N^2n_N.
\]
\end{lemma}

\begin{proof}
Let $D_+=D(\bm n)\setminus\{1\}$ and choose a maximal element $m\in D_+$ for divisibility.  Then no support degree strictly divisible by $m$ occurs, so
\[
w_m=m^2n_m>0,
\]
and hence $w_m=\wt/2$.  There cannot be two distinct maximal elements $m,m'$, since then
\[
\wt\ge n_1+w_m+w_{m'}=n_1+\wt>\wt.
\]
Thus $m$ is the unique maximal element.  Every element of the finite divisibility poset $D_+$ lies below a maximal element, hence every $d\in D_+$ divides $m$.  Therefore $N=m$.

If $1<d<N$ also lies in the support, then
\[
w_d\ge d^2n_d+N^2n_N>\wt/2,
\]
a contradiction.  Thus $D(\bm n)=\{1,N\}$.  Finally
\[
N^2n_N=w_N=\wt/2,
\qquad
\wt=n_1+N^2n_N,
\]
which gives $n_1=N^2n_N$.
\end{proof}

\begin{theorem}[Finite cyclotomic reflection rigidity]\label{thm:finite}
Assume $N\ge2$.  The following are equivalent:
\begin{enumerate}[label=(\roman*)]
\item $\Gamma_{\bm n}$ is finite;
\item $\supp(\bm n)=\{1,N\}$ and $n_1=N^2n_N$;
\item $B_{\bm n}=\frac12(I_N+J_N)$;
\item $\Gamma_{\bm n}\cong W(A_N)\cong S_{N+1}$.
\end{enumerate}
\end{theorem}

\begin{proof}
If $\Gamma_{\bm n}$ is finite, every $R_gR_h$ has finite order, so \Cref{lem:finitepair} implies $w_r/\wt\in\{0,1/2\}$ for every $r>1$.  Apply \Cref{lem:support}.

Conversely, suppose $\supp(\bm n)=\{1,N\}$ and $n_1=N^2n_N$.  Then $\wt=2N^2n_N$ and every nonidentity $g\in C_N$ satisfies $\kappa_{\bm n}(g)=1/2$.  Thus $B_{\bm n}=\frac12(I+J)$.

Let $e_0,\ldots,e_N$ be the standard basis of $\R^{N+1}$ and put
\[
u_i=\frac{e_0-e_i}{\sqrt2},\qquad 1\le i\le N.
\]
Their Gram matrix is $\frac12(I+J)$.  Reflection perpendicular to $u_i$ exchanges $e_0$ and $e_i$.  Hence the generating reflections correspond to the star transpositions $(0\,i)$, which generate $S_{N+1}=W(A_N)$.
\end{proof}

\begin{corollary}[Spectral half-equipartition]\label{cor:equip}
For $N\ge2$, $\Gamma_{\bm n}$ is finite if and only if
\[
\lambda_0=\frac{N+1}{2},\qquad \lambda_k=\frac12\quad(1\le k<N).
\]
Equivalently, $\|P_kv_{\bm n}\|^2=1/(2N)$ for every nontrivial character.
\end{corollary}

\begin{theorem}[Lattice-preserving rigidity]\label{thm:lattice}
For $N\ge2$, $\Gamma_{\bm n}$ is finite if and only if it preserves a full rank lattice in $V_{\bm n}$.
\end{theorem}

\begin{proof}
The finite $A_N$ case preserves the $A_N$ root lattice.  Conversely, suppose a full rank lattice is invariant.  Then $R_gR_h$ is integral in a lattice basis, so its full space trace is an integer.  Since
\[
\Tr(R_gR_h)=N-4+4c(g,h)^2,
\]
we obtain $4c(g,h)^2\in\Z$.  With $c(g,h)\in\Q\cap[0,1)$, this again forces $c\in\{0,1/2\}$.  Apply \Cref{lem:support,thm:finite}.
\end{proof}

\begin{corollary}[Semisimple algebra form of rigidity]\label{cor:semisimple-rigidity}
Let $S=\prod_d M_d(k)^{n_d}$ be a split finite dimensional semisimple algebra with $n_1>0$ and $N=\operatorname{lcm}\{d:n_d>0\}\ge2$.  The reflection group of its cyclotomic metric shadow is finite (equivalently, lattice preserving) if and only if there is an integer $m\ge1$ such that
\begin{equation}\label{eq:semisimple-rigid-form}
S\cong k^{N^2m}\times M_N(k)^m.
\end{equation}
In this case the reflection group is $W(A_N)\cong S_{N+1}$.
\end{corollary}

\begin{proof}
By \Cref{thm:finite,thm:lattice}, finiteness is equivalent to $\supp(\bm n)=\{1,N\}$ and $n_1=N^2n_N$.  Writing $m=n_N$ gives \eqref{eq:semisimple-rigid-form}, and the converse is immediate.
\end{proof}

\subsection{The cases \texorpdfstring{$N=1$ and $N=2$}{N=1 and N=2}}
If $N=1$, then $V_{\bm n}\cong\R$, $B=[1]$, and the unique Householder reflection is $-1$.  Hence
\[
\Gamma_{\bm n}\cong C_2=W(A_1).
\]

For $N=2$,
\[
B=\begin{pmatrix}1&\kappa\\ \kappa&1\end{pmatrix},
\qquad
\kappa=\frac{4n_2}{n_1+4n_2}.
\]
Let $\theta=\arccos\kappa$.

\begin{theorem}[Finite-or-dense dichotomy]\label{thm:dense}
Assume $N=2$ and $n_1,n_2>0$.  If $n_1=4n_2$, then $\Gamma_{\bm n}\cong S_3$.  Otherwise $\theta/\pi\notin\Q$, $\Gamma_{\bm n}\cong D_\infty$, and its image is dense in $O(V_{\bm n})$.
\end{theorem}

\begin{proof}
If $\theta/\pi\in\Q$, the product of the two reflections has finite order.  By \Cref{lem:finitepair}, $\kappa\in\{0,1/2\}$.  Since $n_2>0$, $\kappa>0$, hence $\kappa=1/2$, equivalent to $n_1=4n_2$.  Otherwise the product is an irrational rotation, whose powers are dense in $SO(2)$.  Together with either reflection this gives a dense subgroup of $O(2)$.  Two involutions whose product has infinite order generate $D_\infty$.
\end{proof}

\section{Threefold flops and contraction algebras}\label{sec:flops}

Let $f:X\to Y$ be a flopping contraction from a smooth quasi-projective complex threefold whose exceptional locus is a single rational curve $C\cong\mathbb P^1$.  Let $\ell$ be the scheme theoretic length and let
\[
n_j=n_{j[C]},\qquad 1\le j\le\ell,
\]
be Katz's local genus zero Gopakumar-Vafa invariants \cite{Katz}; the classical threefold flop and length background is developed in \cite{Reid,KatzMorrison}, while the corresponding multiple cover and length structure goes back to Bryan-Katz-Leung \cite{BKL}.  Related singular symplectic flop phenomena, including Ruan cohomological invariance questions, were studied by Chen-Li-Li-Zhao \cite{ChenLiLiZhao2014}.  We set
\[
\bm n_f=(n_1,\ldots,n_\ell).
\]
Toda's theorem \cite{TodaWidth} gives
\begin{equation}\label{eq:toda}
\dim_\C A_{\con}=\sum_{j=1}^{\ell}j^2n_j,
\qquad
\dim_\C A_{\con}^{\ab}=n_1,
\end{equation}
and records $n_j\ge1$ for $1\le j\le\ell$.  Consequently
\[
N_f=\operatorname{lcm}(1,\ldots,\ell),
\qquad
\wt(\bm n_f)=\dim A_{\con}.
\]

\begin{proposition}[Primitive width formula]\label{prop:primitive-width}
For every primitive character of $C_{N_f}$,
\[
\lambda_{\rm prim}(f)=\frac{\dim A_{\con}^{\ab}}{\dim A_{\con}}.
\]
\end{proposition}

\begin{proof}
Combine \Cref{cor:extrema} with \eqref{eq:toda}.
\end{proof}

\begin{theorem}[Higher length infinitude]\label{thm:higherlength}
If $\ell\ge3$, then $\Gamma_f$ is infinite.
\end{theorem}

\begin{proof}
The support contains $1,2,3$.  By \Cref{thm:finite}, a finite group would require support exactly $\{1,N_f\}$, a contradiction.
\end{proof}

If $\ell=1$, then $N_f=1$ and therefore $\Gamma_f\cong W(A_1)$, independently of the value of $n_1$.

Suppose now $\ell=2$.  Then
\[
\dim A_{\con}=n_1+4n_2,
\qquad
\dim A_{\con}^{\ab}=n_1,
\]
and
\[
\kappa_f=\frac{4n_2}{n_1+4n_2}
=1-\frac{\dim A_{\con}^{\ab}}{\dim A_{\con}}.
\]

\begin{theorem}[Half width rigidity for length two]\label{thm:halfwidth}
For a length two irreducible flop, the following are equivalent:
\[
\Gamma_f\text{ is finite},\qquad n_1=4n_2,
\qquad \dim A_{\con}=2\dim A_{\con}^{\ab}.
\]
In the finite case $\Gamma_f\cong W(A_2)\cong S_3$; otherwise $\Gamma_f$ is a dense copy of $D_\infty$ in the corresponding two-dimensional orthogonal group.
\end{theorem}

\begin{proof}
The equivalence of the first two conditions is \Cref{thm:dense}; the second and third are equivalent by \eqref{eq:toda}.  The final statement also follows from \Cref{thm:dense}.
\end{proof}

\subsection{Semisimple shadows and type \texorpdfstring{$D$}{D}}
Suppose a contraction algebra admits a semisimple deformation with generic fibre
\[
S\cong\prod_d M_d(\C)^{n_d}.
\]
For each $r\mid N$ define the semisimple divisibility sector
\[
S^{(r)}=\prod_{r\mid d}M_d(\C)^{n_d}.
\]
Then
\[
\dim S^{(r)}=\sum_{r\mid d}d^2n_d=w_r(\bm n),
\]
and hence
\[
\kappa_{\bm n}(g)=\frac{\dim S^{(\ord g)}}{\dim S}.
\]
Thus the cyclotomic metric is determined by the semisimple block-size data.  This use of semisimple and derived contraction-algebra information is complementary to reconstruction and categorification results such as \cite{HuaToda,HuaKeller,Davison}.  We make no claim here that the sectors $S^{(r)}$ specialize to deformation-independent ideals of the original contraction algebra.

Jelisiejew-Smoktunowicz recently proved Wemyss' semisimple deformation conjecture for types $A$ and $D$ \cite{JS2026}.  In type $D_{n,m}$ their semisimple target has the form
\[
M_2(\C)^{n-1}\times \C^{2m}\qquad (m\le n),
\]
and
\[
M_2(\C)^{n-1}\times \C^{2n+1}\qquad (m>n).
\]
Hence, respectively,
\[
(n_1,n_2)=(2m,n-1),\qquad
(n_1,n_2)=(2n+1,n-1).
\]

\begin{theorem}[Type $D$ classification]\label{thm:Dclassification}
Among the noncommutative $D_{n,m}$ contraction algebras covered by \cite{JS2026} (so that the $M_2$-multiplicity is positive), the cyclotomic reflection group is finite if and only if $(n,m)=(2,2)$.  In that case $(n_1,n_2)=(4,1)$ and $\Gamma\cong W(A_2)$.  Every other case has a dense infinite dihedral reflection group.
\end{theorem}

\begin{proof}
For $m\le n$, finiteness requires $2m=4(n-1)$, hence $m=2(n-1)$.  Together with $m\le n$ in the noncommutative range this forces $(n,m)=(2,2)$.  For $m>n$, the equation $2n+1=4(n-1)$ has no integral solution.  Apply \Cref{thm:dense}.
\end{proof}

\begin{remark}[Coarseness of the shadow]
The cyclotomic shadow records only the Gopakumar-Vafa-visible part of the contraction algebra.  In particular, the examples of Brown-Wemyss with the same Gopakumar-Vafa spectrum but non-isomorphic contraction algebras \cite{BrownWemyss} have identical cyclotomic metric shadows.
\end{remark}

\section{BPS difference Hessians and cyclotomic Riemann-Hilbert problems}\label{sec:RHP}

This section provides an analytic realization of the metric shadow.  The abstract results in \Cref{sec:metric,sec:fourier,sec:rigidity} do not depend on this realization.  For background on BPS/DT wall-crossing and curve counting structures related to the ingredients used below, see \cite{KontsevichSoibelman,PandharipandeThomas,BridgelandHall,MaulikToda,Szendroi}.  The use of profile function/resolvent thermodynamic methods is in the same broad circle of ideas as \cite{LiYagi2021}.  The local Gromov-Witten difference equation problem studied in \cite{LiDiffEq2026} has a different objective: it reconstructs enumerative information from forcing terms rather than constructing the reflection metric considered here.

\subsection{Exact chargewise difference equations}
Fix $\hbar>0$ and put $q=e^{-\hbar}\in(0,1)$.  Introduce the entire charge factor
\begin{equation}\label{eq:Mh}
\mathcal M_{\hbar}(Q)=\prod_{m=1}^{\infty}(1-q^mQ)^m.
\end{equation}
The product converges locally uniformly on the $Q$-plane, since on every compact set
\[
\sum_{m\ge1}m|q|^m|Q|<\infty.
\]
Its zero set is $\{q^{-m}:m\ge1\}$.  For $d\in D(\bm n)$ and $\Re t>0$, define a local logarithm of the positive charge factor by
\begin{equation}\label{eq:Phi+}
\Phi^+_{\hbar,d}(t)=n_d\log\mathcal M_{\hbar}(e^{-dt}).
\end{equation}
In the positive chamber this agrees with
\[
\Phi^+_{\hbar,d}(t)=n_d\sum_{m\ge1}m\log(1-e^{-m\hbar-dt}),
\]
and expanding the logarithm gives
\begin{equation}\label{eq:Phi-series}
\Phi^+_{\hbar,d}(t)
=-n_d\sum_{k\ge1}\frac{e^{-kdt}}{k(2\sinh(k\hbar/2))^2}.
\end{equation}

The product \eqref{eq:Mh} satisfies an exact functional equation.

\begin{lemma}[Charge factor functional equation]\label{lem:Mfunctional}
For all $Q\in\C$ as an identity of entire/meromorphic products,
\begin{equation}\label{eq:Mfunctional}
\frac{\mathcal M_{\hbar}(qQ)\mathcal M_{\hbar}(q^{-1}Q)}{\mathcal M_{\hbar}(Q)^2}=1-Q,
\end{equation}
where the quotient is understood by cancellation at common zeros.
\end{lemma}

\begin{proof}
Reindexing the two shifted products gives
\[
\mathcal M_{\hbar}(q^{-1}Q)
=(1-Q)\prod_{r\ge1}(1-q^rQ)^{r+1}
\]
and
\[
\mathcal M_{\hbar}(qQ)
=\prod_{r\ge2}(1-q^rQ)^{r-1}.
\]
After division by $\mathcal M_{\hbar}(Q)^2$, every factor with $r\ge1$ cancels and only $1-Q$ remains.
\end{proof}

For $a>0$, write
\[
\Delta_aF(t)=F(t+a)-2F(t)+F(t-a),
\qquad
\nabla_a^2F=\frac{\Delta_aF}{a^2}.
\]

\begin{proposition}[Exact charge difference equation]\label{prop:difference}
For $\Re t>\hbar/d$,
\[
\Delta_{\hbar/d}\Phi^+_{\hbar,d}(t)=n_d\log(1-e^{-dt}),
\]
with the logarithm taken on the positive chamber branch.  Equivalently,
\[
\hbar^2\nabla_{\hbar/d}^2\Phi^+_{\hbar,d}(t)
=d^2n_d\log(1-e^{-dt}).
\]
\end{proposition}

\begin{proof}
Under $t\mapsto t\pm\hbar/d$, the variable $Q=e^{-dt}$ is sent to $q^{\pm1}Q$.  Taking a logarithm of \eqref{eq:Mfunctional} in the positive chamber gives the first identity.  The second follows from $\hbar^2/(\hbar/d)^2=d^2$.  Equivalently, one may apply the central difference directly to \eqref{eq:Phi-series}.
\end{proof}

The symmetric completion can now be defined without analytically continuing the logarithmic series.  Let
\[
\Sigma_{\hbar,d}=\left\{\frac{m\hbar}{d}:m\in\Z_{>0}\right\}.
\]
For real $t>0$ with $t\notin\Sigma_{\hbar,d}$, set
\begin{equation}\label{eq:Phisym}
\Phi^{\rm sym}_{\hbar,d}(t)
=\frac{n_d}{2}\left[
\log\left|\mathcal M_{\hbar}(e^{-dt})\right|
+\log\left|\mathcal M_{\hbar}(e^{dt})\right|
\right],
\end{equation}
and extend to $t<0$ by evenness.  The excluded set is precisely the zero lattice of the second charge factor.  It is discrete and plays no role in the continuous interaction extracted below.

\begin{theorem}[Exact symmetric difference Hessian]\label{thm:hessian}
For real $|t|>\hbar/d$ with $|t|\notin\Sigma_{\hbar,d}$,
\begin{equation}\label{eq:hessian}
\hbar^2\nabla_{\hbar/d}^2\Phi^{\rm sym}_{\hbar,d}(t)
=d^2n_d\log\!\left(2\sinh\frac{d|t|}{2}\right).
\end{equation}
\end{theorem}

\begin{proof}
It is enough to take $t>0$.  Apply \eqref{eq:Mfunctional} first to $Q=e^{-dt}$ and then to $Q=e^{dt}$, and take logarithms of absolute values.  This gives
\[
\Delta_{\hbar/d}\Phi^{\rm sym}_{\hbar,d}(t)
=\frac{n_d}{2}\bigl[\log|1-e^{-dt}|+\log|1-e^{dt}|\bigr].
\]
Since $t>0$,
\[
\frac12\bigl[\log(1-e^{-dt})+\log(e^{dt}-1)\bigr]
=\log\!\left(2\sinh\frac{dt}{2}\right).
\]
Multiplying by $d^2$ proves the claim.  The negative half line follows from the even definition.
\end{proof}

\begin{remark}[Meaning of the symmetric completion]
Equation \eqref{eq:Phisym} pairs the same entire charge factor at $Q=e^{-dt}$ and $Q^{-1}=e^{dt}$.  No claim is made that the resulting product is, for an arbitrary higher length flop, the complete geometric Donaldson-Thomas partition function.  The exact statement used in this paper is the functional equation \eqref{eq:Mfunctional} and its symmetric difference Hessian consequence \eqref{eq:hessian}.
\end{remark}

Define
\begin{equation}\label{eq:L}
\mathcal L_{\bm n}(t)=\sum_d d^2n_d\log\!\left(2\sinh\frac{d|t|}{2}\right),\qquad t\ne0.
\end{equation}
Equation \eqref{eq:hessian} gives an exact sectorwise finite-$\hbar$ realization of this interaction away from the discrete zero lattices; the right hand side itself extends smoothly away from $t=0$.

\begin{lemma}[Local singularity]\label{lem:singularity}
As $t\to0$,
\[
\mathcal L_{\bm n}(t)=\wt\log|t|+C_{\bm n}+O(t^2),
\qquad
C_{\bm n}=\sum_dd^2n_d\log d,
\]
and
\[
K_{\bm n}(t):=\mathcal L'_{\bm n}(t)=\frac\wt t+O(t).
\]
\end{lemma}

\begin{proof}
Use
\[
2\sinh\frac{d|t|}{2}=d|t|\left(1+\frac{d^2t^2}{24}+O(t^4)\right)
\]
and differentiate away from $t=0$.
\end{proof}

Thus the singular part of the saddle kernel is the standard Hilbert-transform kernel; the higher sectors modify its regular and nonlocal cyclotomic structure.

\subsection{The associated variational problem}
Let $\rho$ be a compactly supported density of finite logarithmic energy and fixed mass $M$.  Define
\begin{equation}\label{eq:energy}
\mathscr E_{\bm n,V}[\rho]
=\frac12\iint\rho(x)\rho(y)\mathcal L_{\bm n}(x-y)\dd x\dd y
+\int V(x)\rho(x)\dd x.
\end{equation}
We do not claim a large deviation derivation of \eqref{eq:energy} from a complete microscopic DT ensemble.  It is the variational problem naturally associated with the exact difference Hessian interaction.

A sufficiently regular stationary density satisfies
\[
\int\rho(y)\mathcal L_{\bm n}(x-y)\dd y+V(x)=\ell_0
\]
on the interior of its support.  Differentiating yields
\begin{equation}\label{eq:saddle}
\operatorname{PV}\int\rho(y)K_{\bm n}(x-y)\dd y+V'(x)=0,
\end{equation}
where
\begin{equation}\label{eq:K}
K_{\bm n}(t)=\frac12\sum_dd^3n_d\coth\frac{dt}{2}.
\end{equation}

\subsection{Cyclotomic decomposition}
For $\xi\in\mu_N$, set
\[
\acoeff_\xi=\sum_{\xi^d=1}d^2n_d,
\qquad \acoeff_1=\wt.
\]
The elementary identity
\[
\frac{z^d+1}{z^d-1}=\frac1d\sum_{\xi^d=1}\frac{z+\xi}{z-\xi}
\]
gives
\begin{equation}\label{eq:Kcycl}
K_{\bm n}(t)=\frac12\sum_{\xi\in\mu_N}\acoeff_\xi\frac{e^t+\xi}{e^t-\xi}.
\end{equation}
Moreover
\[
\frac{\acoeff_\xi}\wt=\kappa_{\bm n}(\xi)
\]
under the natural identification $\mu_N\cong C_N$.

Introduce $X=e^x$ and the resolvent
\begin{equation}\label{eq:resolvent}
\mathcal R(X)=\int\rho(y)\frac{X+e^y}{X-e^y}\dd y.
\end{equation}
Then \eqref{eq:saddle} becomes
\begin{equation}\label{eq:scalar-nonlocal}
\wt\operatorname{PV}\mathcal R(X)
+\sum_{\xi\ne1}\acoeff_\xi\mathcal R(X/\xi)
=-2V'(\log X).
\end{equation}

\begin{assumption}[Regular cyclotomic one cut regime]\label{ass:onecut}
We assume that the support of $\rho$ is a compact interval, $\rho$ has sufficient H\"older regularity for Plemelj boundary values, the rotated cuts $\Gamma_j=\omega^j\Gamma$ have disjoint interiors, and the external forcing admits compatible single-valued holomorphic continuations near these cuts.
\end{assumption}

Under \Cref{ass:onecut}, define
\[
\mathcal R_j(X)=\mathcal R(\omega^{-j}X),\qquad
\bm{\mathcal R}=(\mathcal R_0,\ldots,\mathcal R_{N-1})^T.
\]
Set
\[
B_{jm}=\frac{\acoeff_{\omega^{m-j}}}\wt.
\]
This is exactly the Gram matrix $B_{\bm n}$.

\begin{theorem}[Cyclotomic reflection Riemann-Hilbert problem]\label{thm:RHP}
Under \Cref{ass:onecut}, on the cut $\Gamma_j$ the vector resolvent satisfies
\begin{equation}\label{eq:affine-jump}
\bm{\mathcal R}_+=R_j\bm{\mathcal R}_-+q_je_j,
\qquad
R_j=I-2e_je_j^TB_{\bm n},
\end{equation}
where, with compatible branches as in \Cref{ass:onecut},
\[
q_j(X)=-\frac4\wt V'\bigl(\log(\omega^{-j}X)\bigr),
\]
is the rotated external forcing.  Moreover
\[
R_j^2=I,\qquad R_j^TB_{\bm n}R_j=B_{\bm n}.
\]
Thus the homogeneous jump representation is precisely the Householder reflection representation of \Cref{sec:rigidity}.
\end{theorem}

\begin{proof}
Plemelj gives $\operatorname{PV}\mathcal R=(\mathcal R_++\mathcal R_-)/2$ on $\Gamma$.  Apply \eqref{eq:scalar-nonlocal} on each rotated cut and collect the $N$ components.  Only the $j$th component jumps, giving \eqref{eq:affine-jump}.  Since $B_{jj}=1$, direct expansion yields $R_j^2=I$ and $R_j^TBR_j=B$.
\end{proof}

\begin{proposition}[Analytic homogenization]\label{prop:homog}
If the forcing vector $\bm q(X)$ is single-valued and holomorphic across all rotated cuts, define
\[
\bm p(X)=\frac12B_{\bm n}^{-1}\bm q(X),
\qquad
\bm Z=\bm{\mathcal R}-\bm p.
\]
Then $\bm Z_+=R_j\bm Z_-$ on $\Gamma_j$.
\end{proposition}

\begin{proof}
Since $B\bm p=\bm q/2$,
\[
(I-R_j)\bm p=2e_je_j^TB\bm p=q_je_j.
\]
Substitution into \eqref{eq:affine-jump} proves the claim.
\end{proof}

The map $u=X^N$ identifies the $N$ rotated cuts in the quotient $u$-plane; equivalently, the vector problem may be viewed as a single-cut problem with $N$ cyclic sheets.  Changing $X$ to $\omega X$ cyclically permutes the components and conjugates the reflections.  This is the sense in which the vectorized problem is cyclotomic.

\section{Folded tangent BPS realization}\label{sec:BPScompare}

We compare the homogeneous jumps of \Cref{thm:RHP} with the tangent geometry of Bridgeland's BPS automorphisms \cite{BridgelandDT}, viewed against the broader wall-crossing formalism of Kontsevich-Soibelman \cite{KontsevichSoibelman}.  The comparison is structural rather than an identification of our auxiliary BPS model with the genuine DT BPS structure of a flop.

\subsection{Tangent form of an uncoupled BPS automorphism}
With Bridgeland's convention, the ray automorphism is
\begin{equation}\label{eq:Bridgeland}
\mathbb S(\ell)^*(x_\beta)
=x_\beta\prod_{Z(\gamma)\in\ell}
(1-x_\gamma)^{\Omega(\gamma)\langle\beta,\gamma\rangle}.
\end{equation}
In local logarithmic coordinates $y_\alpha=\log x_\alpha$,
\[
\delta y_\beta'
=\delta y_\beta
-\sum_\gamma\Omega(\gamma)\langle\beta,\gamma\rangle
\frac{x_\gamma}{1-x_\gamma}\,\delta y_\gamma.
\]

\begin{proposition}[Unipotent tangent form]\label{prop:unipotent}
For a finite uncoupled BPS ray, $D\mathbb S(\ell)=I+N_\ell$ with $N_\ell^2=0$.  In particular $\det D\mathbb S(\ell)=1$.
\end{proposition}

\begin{proof}
Write
\[
(N_\ell\,\delta y)_\beta
=-\sum_{\gamma\in\Gamma_\ell}
c_\gamma\,\langle\beta,\gamma\rangle\,\delta y_\gamma,
\qquad
c_\gamma=\Omega(\gamma)\frac{x_\gamma}{1-x_\gamma}.
\]
For an uncoupled ray, $\langle\gamma,\gamma'\rangle=0$ for all active classes $\gamma,\gamma'$.  Hence
\[
(N_\ell^2\,\delta y)_\beta
=\sum_{\gamma,\gamma'}
c_\gamma c_{\gamma'}
\langle\beta,\gamma\rangle
\langle\gamma,\gamma'\rangle
\delta y_{\gamma'}=0.
\]
Thus $N_\ell^2=0$, so every eigenvalue of $I+N_\ell$ is $1$ and the determinant is $1$.
\end{proof}

\begin{theorem}[Direct linearization obstruction]\label{thm:obstruction}
No cyclotomic reflection $R_j$ is linearly conjugate to the ordinary tangent map of a finite uncoupled Bridgeland BPS ray automorphism.
\end{theorem}

\begin{proof}
By \Cref{prop:unipotent}, the tangent map has determinant $+1$.  A genuine reflection has determinant $-1$.  Determinant is invariant under linear conjugacy.
\end{proof}

\subsection{Shear-fold factorization}
Let $b_j^T=e_j^TB$ and $c_j^T=b_j^T-e_j^T$.  Since $B_{jj}=1$, $c_j^Te_j=0$.  Define
\begin{equation}\label{eq:DjUj}
D_j=I-2e_je_j^T,
\qquad
U_j=I-2e_jc_j^T.
\end{equation}
Then $(U_j-I)^2=0$, $\det U_j=1$, and $D_j^2=I$.

\begin{proposition}[Shear-fold factorization]\label{prop:factor}
For every $j$,
\[
R_j=U_jD_j,
\qquad
D_jU_jD_j=U_j^{-1}.
\]
Consequently $(U_jD_j)^2=I$.
\end{proposition}

\begin{proof}
Using $c_j^Te_j=0$,
\[
U_jD_j=(I-2e_jc_j^T)(I-2e_je_j^T)
=I-2e_j(e_j^T+c_j^T)=I-2e_jb_j^T=R_j.
\]
Since $D_je_j=-e_j$ and $c_j^TD_j=c_j^T$ (because $c_j^Te_j=0$),
\[
D_jU_jD_j=I+2e_jc_j^T=U_j^{-1}.
\]
\end{proof}

We call $D_j$ the \emph{cyclotomic folding involution}.  It is an involution of the tangent model; we do not claim it is itself a Poisson automorphism of the auxiliary BPS lattice.

\subsection{Auxiliary BPS realization of the shear}
Define the integral coefficients
\begin{equation}\label{eq:Ajm}
\acoeff_{jm}=\wt B_{jm}
=\sum_{(\omega^{j-m})^d=1}d^2n_d\in\Z_{\ge0}.
\end{equation}
Fix $j$ and introduce the lattice
\[
\Gamma_j=\Z\beta_j\oplus\bigoplus_{m\ne j}\Z\gamma_m
\]
with skew form
\[
\langle\beta_j,\gamma_m\rangle=1,
\qquad
\langle\gamma_m,\gamma_r\rangle=0.
\]
Choose central charges so that all active $\gamma_m$ lie on one ray and set
\[
\Omega_j(\gamma_m)=\Omega_j(-\gamma_m)=\acoeff_{jm}.
\]
This is a finite integral uncoupled BPS structure.  Set
\begin{equation}\label{eq:qstar}
q_*=\frac2{\wt+2},
\qquad
\frac{q_*}{1-q_*}=\frac2\wt.
\end{equation}
Choose a regular twisted torus point $\xi_*$ with $x_{\gamma_m}(\xi_*)=q_*$ for all $m\ne j$.  Such a point exists: the active classes $\gamma_m$ span an isotropic direct summand, so their twisted multiplication has no sign obstruction and the values $x_{\gamma_m}=q_*$ may be prescribed simultaneously; one then chooses the remaining coordinate $x_{\beta_j}$ arbitrarily away from the singular locus.  Since the active set is finite, the support property of the auxiliary BPS structure is automatic.

\begin{theorem}[Folded tangent BPS realization]\label{thm:folded}
For the auxiliary BPS structure above,
\[
D_{\xi_*}\mathbb S_j=U_j.
\]
Hence
\[
R_j=(D_{\xi_*}\mathbb S_j)D_j.
\]
\end{theorem}

\begin{proof}
Apply the logarithmic differential of \eqref{eq:Bridgeland} to the coordinate $y_{\beta_j}$.  Using \eqref{eq:qstar}, the coefficient of $\delta y_{\gamma_m}$ is
\[
-\Omega_j(\gamma_m)\frac{q_*}{1-q_*}
=-\acoeff_{jm}\frac2\wt=-2B_{jm}.
\]
The active $\gamma_m$ coordinates are fixed because the active classes are mutually orthogonal.  Thus the tangent matrix is exactly $U_j$.  Apply \Cref{prop:factor}.
\end{proof}

\begin{remark}[Scope and choices]
The matrix $U_j$ and the multiplicities $\acoeff_{jm}$ are canonically determined by the cyclotomic data.  The auxiliary BPS realization of that matrix still involves harmless choices of central charges and of a regular torus point with the prescribed active coordinates.  The auxiliary BPS multiplicities are the cyclotomic divisibility widths, not the original geometric DT spectrum.  Therefore \Cref{thm:folded} is a universal tangent realization of the reflection system and not an equivalence between the genuine DT Riemann-Hilbert problem of a flop and the cyclotomic problem of \Cref{sec:RHP}.  Bridgeland's genuine resolved-conifold problem is treated in \cite{BridgelandConifold}.
\end{remark}

\subsection{The rank two model}
When $N=2$,
\[
B=\begin{pmatrix}1&\kappa\\ \kappa&1\end{pmatrix},
\qquad
U_1=\begin{pmatrix}1&0\\-2\kappa&1\end{pmatrix},
\qquad
D_1=\begin{pmatrix}1&0\\0&-1\end{pmatrix},
\]
so
\[
R_1=\begin{pmatrix}1&0\\-2\kappa&-1\end{pmatrix}.
\]
For the finite spectrum $(n_1,n_2)=(4,1)$, one has $\wt=8$, $\kappa=1/2$, and $q_*=1/5$.  For $(5,1)$ one has $\wt=9$, $\kappa=4/9$, and $q_*=2/11$.  In both cases the tangent factor remains unipotent; the finite/dense transition is produced after folding.

\section*{Discussion}
For a split semisimple algebra, the metric shadow together with the total dimension recovers the complete Wedderburn type by \Cref{thm:recover-semisimple}.  For a nonsemisimple contraction algebra, however, the construction retains only the Gopakumar-Vafa visible part of the information.  It therefore should not be interpreted as a classification of flops.  Its value lies instead in producing, from local enumerative data, a canonical positive metric and a rigid reflection representation with two independent realizations: one through contraction algebra width data and one through an associated singular integral/Riemann--Hilbert model.

Several questions remain open.  First, it would be interesting to derive the variational functional of \Cref{sec:RHP} from a microscopic DT or random partition ensemble by a genuine large deviation argument, in the spirit of Nekrasov-Okounkov \cite{NekrasovOkounkov}.  Second, one may ask whether the folded tangent comparison of \Cref{sec:BPScompare} can be related more directly to the genuine Bridgeland DT BPS structure of a flop.  Third, in higher cyclotomic rank, the closure of an infinite reflection group inside $O(B_{\bm n})$ deserves further study; unlike the rank two case, no general density statement is claimed here.

\section*{Declaration of competing interest}
The author declares that he has no known competing financial interests or personal relationships that could have appeared to influence the work reported in this paper.

\section*{Data availability}
No data was used for the research described in this article.

\end{document}